\documentclass[a4paper,12pt]{article}
\usepackage{amssymb}
\usepackage{amsthm}
\usepackage{amsmath}
\usepackage{graphicx}
\usepackage{psfrag}
\usepackage{array,hhline}
\usepackage{enumitem}
\usepackage{caption}
\usepackage{subcaption}

\newtheorem{lemma}{Lemma}[section]
\newtheorem{theorem}[lemma]{Theorem}
\newtheorem{proposition}[lemma]{Proposition}

\newtheorem{corollary}[lemma]{Corollary}
\theoremstyle{definition}
\newtheorem{definition}[lemma]{Definition}

\numberwithin{equation}{section}
\numberwithin{figure}{section}

\newcommand{\Aset}{\mathcal{A}}
\newcommand{\Bset}{\mathcal{B}}

\newcommand{\Lset}{\mathcal{L}}
\newcommand{\Mset}{\mathcal{M}}

\newcommand{\Xset}{\mathcal{X}}

\begin{document}

\title{\large{On the usage of $2$-node lines in $n$-correct and $GC_n$-sets}}

\author{H. Hakopian, \  G. Vardanyan, \  N. Vardanyan}

\author{H. Hakopian, G. Vardanyan, N. Vardanyan\\ \\
 \textit{Yerevan State University, Yerevan, Armenia}\\   \textit{Institute of Mathematics of NAS RA}}
\date{}
\maketitle

\begin{abstract} 
An n-correct set $\mathcal{X}$ in the plane is a set of nodes admitting unique
 interpolation with bivariate polynomials of total degree at most $n$.
A $k$-node line is a line passing through exactly $k$ nodes of $\mathcal{X}.$ A line can pass through at most $n+1$ nodes of an $n$-correct set. An $(n+1)$-node line is called maximal line. We say that a node $A\in\Xset$ uses a line $\ell,$ if $\ell$ is a factor of the fundamental polynomial of the node $A.$

One of the main problems we study in this paper is to determine the maximum possible number of used $2$-node lines that share a common node $B$ in an $n$-correct set $\mathcal{X}$.
We show that this number equals $n$. Moreover,  if there are $n$ such $2$-node lines, then $\mathcal{X}$ contains exactly $n$ maximal lines not passing through the common node $B$. Furthermore, if $\mathcal{X}$ is a $GC_n$-set, then there exists an additional maximal line passing through $B$. In this case $\mathcal{X}$ has exactly $n+1$ maximal lines and therefore  is a Carnicer–Gasca set of degree $n$. 

Note that Carnicer–Gasca sets of degree $n$ with a prescribed set of $n$ used $2$-node lines can be readily constructed.
\end{abstract}




\section{Introduction}
Denote by $\Pi_n$ the space of bivariate polynomials of total degree
$\le n.$ We set
$$
N:=\dim \Pi_n,\ \hbox{where}\ N=(1/2){(n+1)(n+2)}.
$$
\noindent Consider a set of $s$ distinct nodes in the plane:
$$\Xset:=\Xset_s=\{ (x_1, y_1), \dots , (x_s, y_s)\}.$$

The problem of finding a polynomial $p \in \Pi_n$ satisfying the conditions
\begin{equation}\label{int cond}
p(x_i, y_i) = c_i, \ \ \quad i = 1, 2, \dots s  ,
\end{equation}
for a given data $\bar c:=\{c_1, \dots, c_s\}$ is called \emph{interpolation problem.}
\begin{definition}
A set of nodes $\Xset_s$ is called
$n$-\emph{correct} if for any data $\bar c$ there exists a
unique polynomial $p \in \Pi_n$, satisfying the conditions
\eqref{int cond}.
\end{definition}
A necessary condition of
$n$-correctness is: $\#\Xset_s=s = N.$

Denote by $p|_\Xset$ the restriction of $p$ on $\Xset.$ The following proposition is well-known (see e.g. \cite{HJZ09b}
Prop. 1.2):
\begin{proposition} \label{correctii}
A set of nodes $\Xset$ with $\#\Xset=N$ is $n$-correct if and only if
$$p \in \Pi_n,\ p|_\Xset=0\implies p = 0.$$
\end{proposition}
A polynomial $p \in \Pi_n$ is called an $n$-\emph{fundamental polynomial}
for $ A \in \Xset$ if
$$ p|_{\Xset\setminus\{A\}}=0\ \hbox{and}\ p(A)=1.$$
This polynomial is denoted by $p_A^\star=p_{A,\Xset}^\star.$

Let us mention that sometimes it is convenient to call $n$-fundamental also a polynomial $p\in \Pi_n$ satisfying the conditions
$$ p|_{\Xset\setminus\{A\}}= 0\ \hbox{and}\ p(A)\neq 0,$$
since it equals a non-zero constant times $p_A^\star.$

\begin{definition}
A set of nodes $\Xset_s$ is called $n$-\emph{independent} if each node
has $n$-fundamental polynomial. 
\end{definition}
Fundamental polynomials are linearly independent. Therefore a
necessary condition of $n$-independence is $\#\Xset_s=s \le N.$

A \emph{plane algebraic curve} of degree $n,\ n\ge 1,$ is the zero set of a bivariate polynomial $p$ of degree $n.$  To simplify notation, we shall use the same letter,  say $p$,
to denote the polynomial $p$ and the curve given by the equation $p(x,y)=0$.
In particular, by $\ell$ we denote a linear
polynomial from $\Pi_1$ and the line defined by the equation
$\ell(x, y)=0.$

Let us mention that in the expressions like ${\mathcal X \cap p},\ \Xset\setminus p,$ by the polynomial $p\in\Pi_k$ we mean its zero set. 
\begin{definition} Let $\Xset$ be an $n$-correct set.
We say that a node $A\in\Xset$ \emph{uses a curve} $q$ of degree $k\le n,$ if $q$ is a component of the fundamental polynomial of the node $A:$ $$ p^\star_A=q r,\ \hbox{where}\  r\in\Pi_{n-k}.$$
\end{definition}

\subsection{$\bf{GC_n}$, Chung-Yao, and Carnicer-Gasca sets}\label{CYCG}

Let us now consider a particular type of $n$-correct sets (see \cite{HV}) that satisfy a so-called \emph{geometric characterization} (GC) property introduced by K.C. Chung and T.H. Yao \cite{CY77}:
\begin{definition}[\cite{CY77}]
An $n$-correct set ${\mathcal X}$ is called a \emph{$GC_n$-set }
 if  the
$n$-fundamental polynomial of each node is a
product of $n$  linear factors.
\end{definition}
Thus, each node of a $GC_n$-set uses exactly $n$ lines.

The following proposition is well-known (see e.g. \cite{HJZ09b}
Prop. 1.3):
\begin{proposition}\label{prp:n+1ell}
Suppose that a polynomial $p \in
\Pi_n$ vanishes at $n+1$ points of a line $\ell.$ Then, we have that
$
p = \ell  q  ,\ \text{where} \ q\in\Pi_{n-1}.
$
\end{proposition}
We say, that a line $\ell$ is a $k$-\emph{node line} if it passes through exactly $k$ nodes of $\Xset.$
The above proposition implies that at most $n+1$ nodes of an $n$-independent set  can be collinear. An $(n+1)$-node line $\ell$ is called a \emph{maximal line} (C. de Boor, \cite{dB07}).

Denote by $\Mset(\Xset)$ the set of maximal lines of an $n$-correct set $\Xset.$

Below some basic properties of maximal lines are presented.

\begin{proposition}[\cite{CG01}, Prop.2.1] \label{hatk} Let $\Xset$ be an $n$-correct set. Then:
\begin{enumerate}
\item
Any two maximal lines intersect at a node of $\Xset.$
\item
Any three maximal lines are not concurrent.
\item
$\#\Mset(\Xset)\le n+2.$
\end{enumerate}
\end{proposition}

Let a set $\Lset=\{\ell_0,\ldots,\ell_{n+1}\}$ of $n+2$ lines be in \emph{general position}, that is 
no two lines of $\Lset$ are parallel, and 
no three lines of $\Lset$ are concurrent.

Then the set $\Xset$ of ${n+2\choose 2}$ intersection points
of these lines is called a
Chung-Yao set of degree $n$ (see \cite{CY77}):
\begin{equation}\label{Aij}\Xset= \{A_{ij}: 0\le i<j\le n+1\},\ \hbox{where}\ A_{ij}=\ell_i\cap\ell_j.\end{equation} 
 
Note that each node $A:=A_{ij}$ belongs to $2$ lines: $\ell_i$ and $\ell_j,$ and the product of the remaining $n$ lines gives $p^\star_{A}.$ Thus $\Xset$ is a $GC_n$-set. It is easily seen that  the lines $\ell_i,\ i=0,\ldots,n+1,$ are maximal for $\Xset.$ In view of Proposition \ref{hatk}, (iii), no other line is maximal for the set $\Xset,$ that is, $\Lset=\Mset(\Xset).$

On the other hand if $\Xset$ is an $n$-correct set and $\#\Mset(\Xset)= n+2$ then $\Xset$ is a Chung-Yao set of degree $n.$ Indeed, assume that $\Mset(\Xset)= \{\ell_0,\ldots,\ell_{n+1}\}$ is the set of maximal lines of $\Xset.$ Then, by using Proposition \ref{hatk}, (i),(ii), we readily get \eqref{Aij}.

Thus 
 an $n$-correct set $\Xset$ is a Chung-Yao set of degree $n$ if and only if $\#\Mset(\Xset)= n+2.$

Now, let a set $\Lset=\{\ell_0,\ldots,\ell_n\}$ of $n+1$ lines be in general position (see Fig.\ref{Fig1}, for $n=3$).
Then the set of ${n+1\choose 2}$ intersection points $\{A_{ij}: 0\le i<j\le n\}$
of these lines together with a set of $n+1$ other noncollinear points $\{A_i\in\ell_i: 0\le i\le n\},$ is called a Carnicer-Gasca set of degree $n$ (see \cite{CG01}, prop. 2.4; \cite{HV}, Section 3). 
\begin{figure}
\begin{center}
\includegraphics[width=8.0cm,height=5.cm]{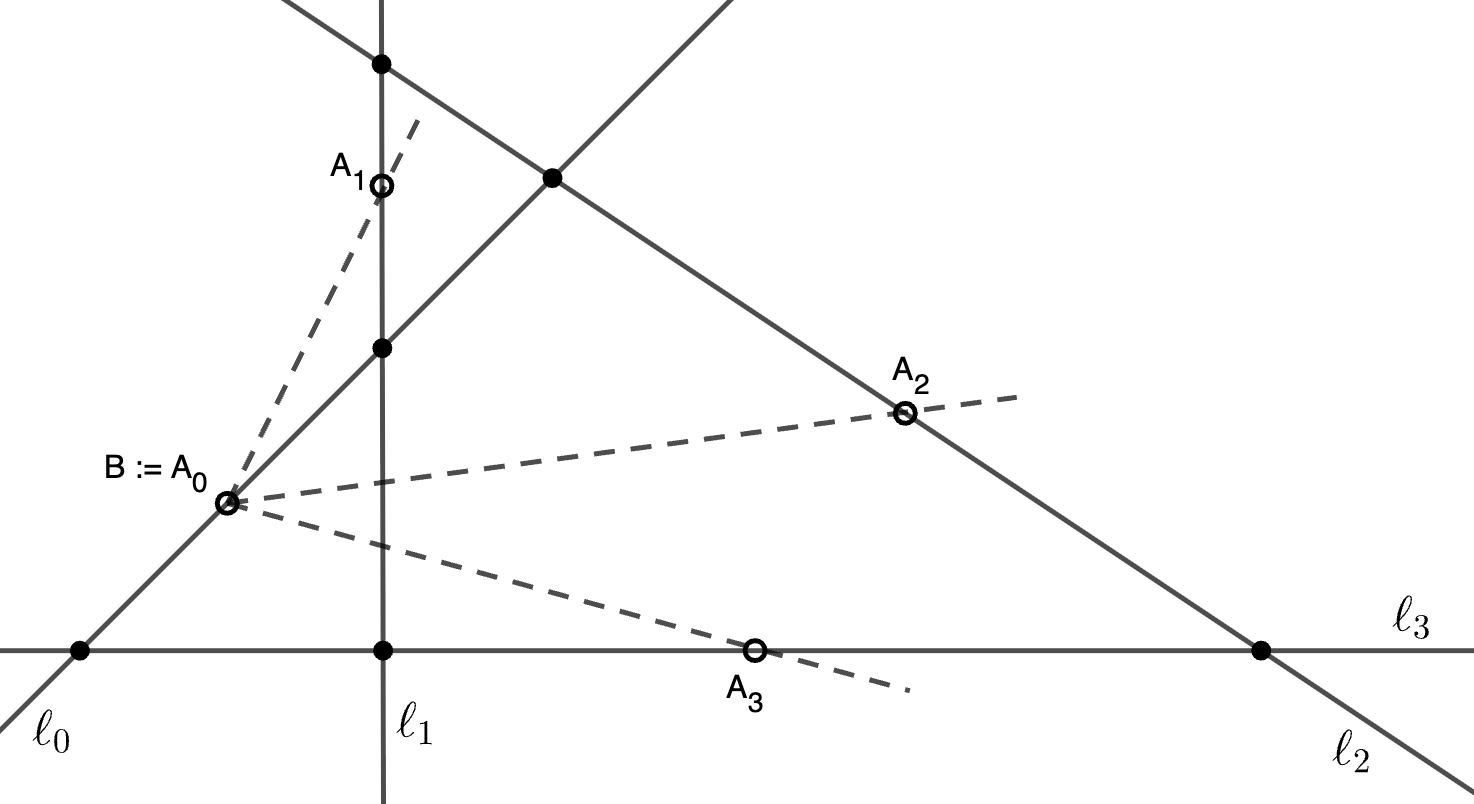}
\end{center}
\caption{A Carnicer-Gasca set of degree $3$} \label{Fig1}
\end{figure}

Denote by $\ell_{AB}$ the line between the points $A,B.$

Each node  $A_i$ here belongs to a single line: $\ell_i,$ and the product of the remaining $n$ lines gives $p^\star_{A_i}.$ Next, each node $A:=A_{ij}$ belongs to two lines: $\ell_i$ and $\ell_j,$ and the product of the remaining $n-1$ lines and the line $\ell_{A_i,A_j}$ gives $p^\star_{A}.$ Hence $\Xset$ is a $GC_n$-set. Note that the lines $\ell_i,\ i=0,\ldots,n,$ are maximal for $\Xset.$  Thus, in a Carnicer-Gasca set, in addition to the maximal lines, the lines $\ell_{A_i,A_j}$ are also used (by the nodes $A_{ij},$ respectively).

It can be checked readily that  
 an $n$-correct set $\Xset$ is a Carnicer-Gasca set of degree $n$ if and only if $\#\Mset(\Xset)= n+1.$

\subsection{Maximal curves}

Set $$d(n, k) := N_n - N_{n-k} =(n+1)+n+(n-1)+\cdots+(n-k+2)= (1/2) k(2n+3-k).$$

The following is a  generalization of Proposition \ref{prp:n+1ell}.
\begin{proposition}[\cite{Raf}, Prop. 3.1]\label{maxcurve}
Let $\Xset$ be an $n$-correct set and $q$ be an algebraic curve of degree $k \le n$ with no multiple components. Then, the following hold:

$(i)$ Any subset of $q$ containing more than $d(n,k)$ nodes of $\Xset$ is
$n$-dependent.

$(ii)$ Any subset $\mathcal S$ of $q$ containing exactly $d(n,k)$ nodes of $\Xset$ is $n$-independent if and only if
\vspace{-.5cm}

$$\quad p\in {\Pi_{n}}\ \text{and}\ \ p|_{{\mathcal S}} = 0 \implies  p = qr,\ \hbox{where}\ r \in \Pi_{n-k}.$$
\end{proposition}
\noindent Thus at most $d(n,k)$ $n$-independent nodes lie in a curve $q$ of degree $k \le n$. This motivates the following
\begin{definition}\label{def:maximal}
Let $\Xset$ be an $n$-correct set of nodes. A curve of degree $k \le n$ passing through $d(n,k)$ nodes
of $\mathcal X$ is called maximal curve for $\Xset.$
\end{definition}
Clearly, a maximal curve cannot have a multiple component.

The following proposition gives a characterization of the maximal curves:
\begin{proposition}[\cite{Raf}, Prop. 3.3] \label{maxcor}
Let $\Xset$ be an $n$-correct set. Then
a curve $f\in\Pi_k,$ where $k\le n,$ is a maximal curve for $\Xset$ if and only if 
it is used by each node of the set $\Xset\setminus f.$
\end{proposition} 

Note that this, in view of the Lagrange formula, follows directly from Proposition \ref{maxcurve}. 

One readily gets from here:
\begin{corollary}\label{maxmax}
Let $\Xset$ be an $n$-correct set. Then
a curve $f\in\Pi_k,$ where $k\le n,$ is a maximal curve for $\Xset$ if and only if the set $\Xset\setminus f$ is an $(n-k)$-correct set. 
Moreover, if $\Xset$ is a $GC_n$-set then $\Xset\setminus q$ is a $GC_{(n-k)}$-set.
\end{corollary}

One also obtains readily that a maximal curve $f$ of degree $k\le n,$ for  a $GC_n$-set, is a product of $k$ distinct lines (see \cite{HVV}, Prop. 5.1). 

Next result concerns components of a maximal curve. We will use it in the special case of the line-components.
\begin{proposition}\label{maxcomp} Let $\Xset$ be an $n$-correct set  and
a curve $q$ of degree $k\le n,$ be a maximal curve. Suppose that a curve $s$ of degree $m\le k,$ is a component of $q,$ that is $q=s r, \deg r=k-m.$
Then $s$ passes through  at least $d(n-k+m,m)$ nodes of $\Xset \setminus r.$ Moreover, $s$ passes through  exactly $d(n-k+m,m)$ nodes of $\Xset \setminus r$ if and only if $r$ is a maximal curve of degree $k-m.$ 
\end{proposition}

\begin{proof} Assume, by way of contradiction, that   $s$ passes through  less than $d(n-k+m,m)$ nodes of $\Xset \setminus r.$
Then, since $q=s r$ passes through $d(n,k)$ nodes of $\Xset,$ we get that the curve $r, \deg r=k-m,$ passes through more than $d(n,k)-d(n-k+m,m)=d(n,k-m)$ nodes of $\Xset,$ which is a contradiction.

Now, if $s$ passes through  exactly $d(n-k+m,m)$ nodes of $\Xset \setminus r$ then above calculation shows that  $r$ passes through exactly $d(n,k)-d(n-k+m,m)=d(n,k-m)$ nodes of $\Xset.$ Hence, $r$ is a maximal curve of degree $k-m.$ And, vice versa, if $r$ passes through exactly $d(n,k-m)$ nodes of $\Xset$ then we similarly get that  $s$ passes through  exactly $d(n-k+m,m)$ nodes of $\Xset \setminus r.$  
\end{proof}

In the sequel we will need the following 
\begin{theorem} [\cite{HakTor}, Thm. 4.1] \label{HT1} Assume that $\Xset$ is an n-independent set of
$d(n; k-1)+2$ nodes lying in a curve of degree $k\le n.$ Then the curve
is determined uniquely by these nodes.
\end{theorem}

\subsection{Special triplets}

Let us start with the following
\begin{definition}
Let $\Xset$ be an $n$-correct set of nodes. A triplet of nodes $\{A,B,C\}\subset\Xset$ is called a \emph{special triplet} if there exists a maximal curve $f$ of degree $n-1$ such that 
\begin{equation} \label{020} \Xset\setminus f =\{A,B,C\}.\end{equation}
\end{definition}
We call the curve $f$ mentioned above the maximal curve associated with the special triplet $\{A, B, C\}$.

An example of a special triplet in a $GC_5$-set  is presented in Fig.\ref{Fig2}. The associated maximal curve in this case is:
$f(x,y)=xy(y-1)(x+y-5).$

A triplet of  the $GC_5$-set presented in Fig.\ref{Fig3} is non-special. Indeed, assume by way of contradiction, that the triplet is special and $f\in\Pi_4$ is the associated maximal curve. Then, in view of Proposition \ref{prp:n+1ell}, we get that $f(x,y)=xy(x+y-5)\ell(x,y),$ where $\ell\in\Pi_1.$ It is easily seen also that $\ell$ vanishes at $(1,2), (2,1)$ and $(2,2).$ Therefore $\ell=0$ and $f=0,$ which is a contradiction. 

\begin{figure}
     \centering
     \begin{subfigure}[b]{0.3\textwidth}
         \centering
         \includegraphics[width=\textwidth]{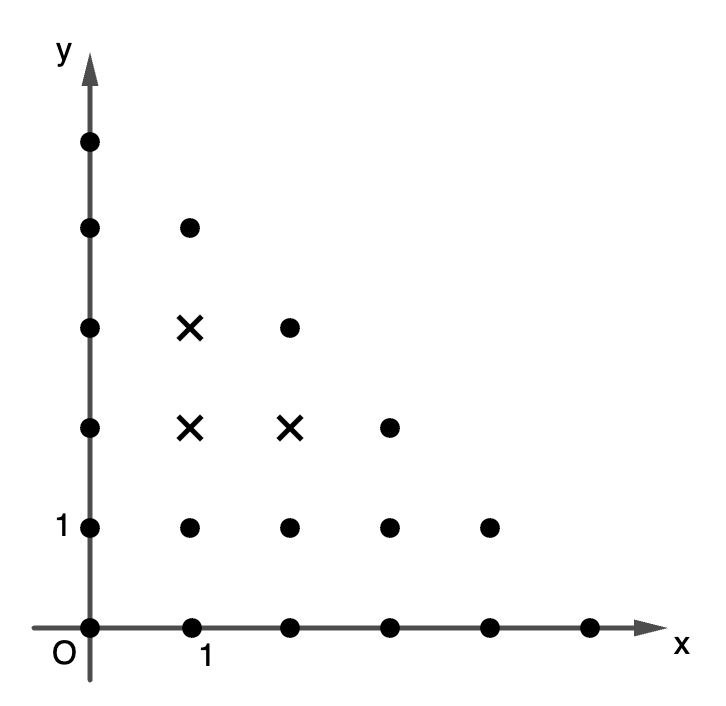}
         \caption{A special triplet}
         \label{Fig2}
     \end{subfigure}
     \hfill
     \begin{subfigure}[b]{0.3\textwidth}
         \centering
         \includegraphics[width=\textwidth]{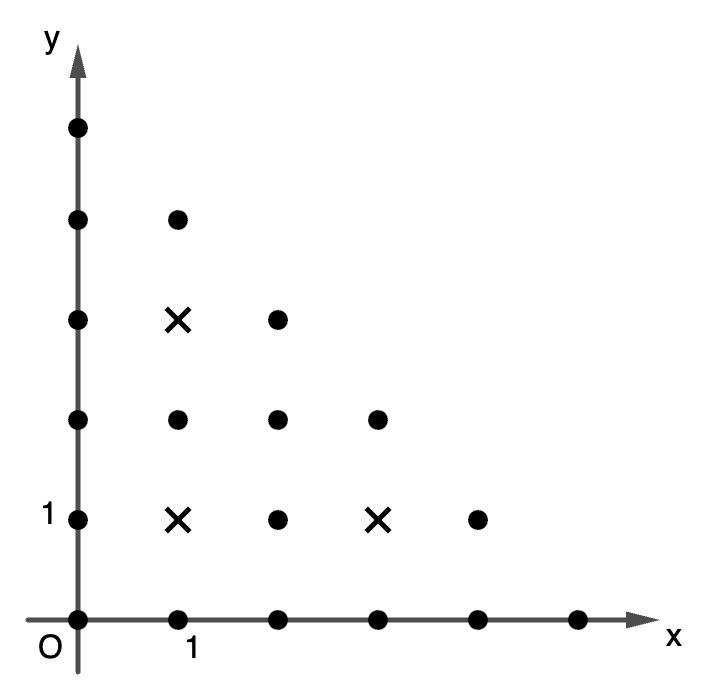}
         \caption{A non-special tripliet}
         \label{Fig3}
     \end{subfigure}
          \caption{Two triplets in a $GC_5$-set}
        \label{fig2}
\end{figure}

\begin{proposition} \label{sptr} Let $\Xset$ be an $n$-correct set of nodes and $\{A,B,C\}\subset\Xset$ be a special triplet. Then the nodes $A,\ B,$ and $C$ are noncollinear, and use

(i) the lines $\ell_{BC},\ \ell_{CA},$ and $\ell_{AB},$ respectively,

(ii) the associated maximal curve:
\begin{equation}\label{ABC} p^\star_{A,\Xset}=\ell_{BC} f, \ \ p^\star_{B,\Xset}=\ell_{CA} f,\ \ \hbox{and} \ \  p^\star_{C,\Xset}=\ell_{AB} f.\end{equation}
\end{proposition}
\begin{proof} Let $f$ be the associated with $\{A,B,C\}$ maximal curve. Assume, by way of contradiction, that $A,B,C$ are collinear, i.e., belong to some line $\ell.$ Then the polynomial $\ell f\in\Pi_n$ vanishes on $\Xset,$ which contradicts Proposition \ref{correctii}. Next, in view of \eqref{020}, we readily obtain \eqref{ABC}.
\end{proof}
Note that the non-special triplet in Fig.\ref{Fig3} is non-collinear and satisfies the condition (i) of Proposition \ref{sptr}.

\begin{proposition} \label{sptr2} Let $\Xset$ be an $n$-correct set of nodes. Let also $\{A,B,C_1\}$ and $\{A,B,C_2\}$ be special triplets, where $A,B,C_1,C_2 \in\Xset$. Then $C_1=C_2.$
\end{proposition}
\begin{proof} Let $f_i\in\Pi_{n-1}$ be the maximal curve associated with $\{A,B,C_i\},\ i=1,2.$ Then $f_1$ and $f_2$ pass through $N-4$ nodes of the set
$\Xset\setminus \{A,B,C_1,C_2\}.$ On the other hand we have that $N-4=d(n,n-2)+2.$
From here, in view of Theorem \ref{HT1}, we conclude that $f_1= f_2$ and therefore $C_1=C_2.$ 
\end{proof}

Below we describe the case when a special triplet $\{A, B, C\}$ originates from a  2-node line $\ell_{AB}$ used by $C.$
\begin{lemma}\label{cor:2ket}
Let $\Xset$ be an $n$-correct set and $\ell_{AB}$ be a $2$-node line used by a node $C,$ where $A, B,C\in\Xset.$ Then $\{A,B,C\}$ is a special triplet. 
\end{lemma}
\begin{proof}
The node $C$ uses the  line $\ell_{AB}$ means that
\begin{equation}\label{01}p^\star_{C,\Xset}=\ell_{AB} f,\ \hbox{where}\  f\in\Pi_{n-1}.
\end{equation}
Since $\ell_{AB}$ is a $2$-node line we conclude from here that the relation \eqref{020} takes place. Therefore $\{A,B,C\}$ is a special triplet and $f$ is the associated maximal curve of degree $n-1$.
\end{proof}
Note that the triplet in Fig.\ref{Fig2} does not originate from a $2$-node line. Indeed, the three lines through the vertices of the triplet are $5,5$ and $4$-node lines.

Next, we obtain the following result of Carnicer and Gasca:
\begin{corollary}[\cite{CG03}, Prop. 4.2] \label{2nodeuse} Let $\Xset$ be an $n$-correct set. Then any $2$-node line can be used by at most one node.
\end{corollary}
Indeed, assume that a $2$ node line $\ell_{AB}$ is used by two nodes $C_1,C_2\in\Xset.$ Then, by Lemma \ref{cor:2ket}, we get that $\{A,B,C_1\}$ and $\{A,B,C_2\}$ are special triplets. Now, by  Proposition \ref{sptr2}, we obtain that $C_1=C_2.$

One of the main problems we study in this paper is determining the maximum possible number of used $2$-node lines that share a common node $B$ in an $n$-correct set $\mathcal{X}$.

We show that this number equals $n$. Furthermore, if there are exactly $n$ such $2$-node lines, then $\mathcal{X}$ contains precisely $n$ maximal lines not passing through the common node $B$. Furthermore, if $\mathcal{X}$ is a $GC_n$-set, then there exists an additional maximal line passing through $B$. Consequently, in the latter case, $\mathcal{X}$ has exactly $n+1$ maximal lines and is a Carnicer-Gasca set of degree $n$. Note that $\mathcal{X}$ cannot have $n+2$ maximal lines and become a Chung-Yao set since there is a $2$-node used line.

Next, let us construct a Carnicer–Gasca set of degree $n$ with a set of $n$ used $2$-node lines.

To this end, let us consider  a Carnicer–Gasca set  $\mathcal{X}$ of degree $n$ with  the maximal lines $\ell_0,\ldots,\ell_n.$ Then the intersection nodes $A_{ij}$  are determined as described earlier in Section \ref{CYCG}. Next let us choose the remaining nodes $A_i \in \ell_i$, with $i=0, \ldots, n$. We start with a node $B := A_0 \in \ell_0$. Choose the remaining nodes $A_i \in \ell_i,\ i=1,\ldots,n,$ so that the $n$ lines $\ell_{A_iB}$ are $2$-node lines (see Fig. \ref{Fig1} for $n=3$). 
By the earlier observation in Section \ref{CYCG} these lines are used by the nodes $A_{0i},\ i=1,\ldots,n,$ respectively. Thus we have exactly $n$ used $2$-node lines that share the node $B\in\Xset.$

\section{Main results}
We are now in a position to present a main result.
\begin{theorem}\label{th1sptr}
Let $\Xset$ be an $n$-correct set and $\mu:=\mu_k$ be a maximal curve of degree $k\le n.$ Hence the set $\Bset:=\Bset_{n-k}=\Xset\setminus \mu$ is an $(n-k)$-correct set. Denote by $\Aset:=\Aset_{k}=\Xset\cap \mu.$\\ Then a triplet of nodes $\{A,B,C\},$ where $A \in \Aset,\ B\in \Bset,$ and $C \in \Xset,$ is a special triplet if and only if the following two conditions hold:
\begin{enumerate}
\setlength{\itemsep}{0mm}
\item
The line $\ell_1:=\ell_{AC}$ is a component of $\mu:\ \mu=\ell_1\mu_{k-1},$ in particular $B\notin \ell_1$ and $C\in \Aset.$  
\item
The set $\widetilde \ell_1:=\{M\in{\ell_1}\cap  \Xset: \mu_{k-1}(M)\neq 0\}$ contains exactly $(n-k)+2$ nodes of $\mathcal X\setminus \mu_{k-1}.$ Namely, it contains

1)  the nodes $A$ and $C,$ 

2) exactly $n-k$ other nodes,  
which are the zeros of $p^\star_{B,\Bset}\in\Pi_{n-k}$ on the line $\ell_1.$ 

The latter $n-k$ zeros are therefore real and distinct.
\end{enumerate}
Moreover, the conditions \emph{(i) and (ii)} imply the following conditions:
\begin{enumerate}
\setcounter{enumi}{2}
\setlength{\itemsep}{0mm}
\item
The curve $\mu_{k-1}:=\frac{1}{\ell_1}\ \mu$ is a maximal curve of degree $k-1.$ 

\item
The set $\Bset_{n-k+1}:=\Xset\setminus\mu_{k-1}=\Bset\cup \widetilde\ell_1$ is an  $(n-k+1)$-correct set, for which $\ell_1$ is a maximal line.  
\item
$\Aset_{k-1}:=\Xset\cap\mu_{k-1}=\Aset\setminus \widetilde\ell_1.$
\item 
$\{A,B,C\}$ is a special triplet also in the $(n-k+1)$-correct set $\Bset_{n-k+1},$ with the associated maximal curve  $p^\star_{B,\Bset}$.
 
\item
If $\{A',B,C'\}\subset\Xset$  is a special triplet, and $A'\in \widetilde\ell_1,$ then $\{A',B,C'\}=\{A,B,C\},$ i.e., $\{A',C'\}=\{A,C\}.$
\item
If $\Xset$ is a $GC_n$-set then $\Bset_{n-k+i}$ is a $GC_{(n-k+i)}$-set, where $i=0,1.$
\end{enumerate}
\end{theorem}

\begin{proof}
Assume that $\{A,B,C\},$ where $A \in \Aset,\ B\in \Bset,$ and $C \in \Xset,$ is a special triplet.
Then Lemma \ref{sptr} implies that
\begin{equation}\label{202}p^\star_{B,\Xset}=\ell_1 f,
\end{equation}
where $f$ is the  maximal curve of degree $n-1$ associated with $\{A,B,C\}:$
\begin{equation}\label{0207}\Xset\setminus f=\{A,B,C\}.\end{equation} On the other hand, we have that
\begin{equation}\label{222}p^\star_{B,\Xset}=p^\star_{B,\Bset}\ \mu.
\end{equation}
In view of \eqref{202} and \eqref{222} the line $\ell_1$ is either a component of $\mu$  or a component of  $p^\star_{B,\Bset}.$ In the latter case we obtain that $f=q\mu,$ where
$q:=p^\star_{B,\Bset}/\ell_1\in\Pi_{n-k-1}.$ Therefore, we get 
$$f(A)=q(A)\mu(A)=0,$$  
which contradicts \eqref{0207}.

Thus $\ell_1=\ell_{AC}$ is a component of $\mu$ and (i) is proved.

Denote $\mu_{k-1}:=\frac{1}{\ell_1}\ \mu_k,\ \mu_{k-1}\in\Pi_{k-1}.$ 
Now, from the relations \eqref{202} and \eqref{222} we obtain that 
\begin{equation}\label{030}f=p^\star_{B,\Bset}\ \mu_{k-1}.\end{equation}

Then, let us consider the set of nodes $\widetilde\ell_1$ defined in (ii).
 Let us verify first that $A,C\in \widetilde\ell_1.$ 

Indeed, in view of \eqref{0207}, we have that $f$ does not vanish on $A$ and $C.$ Therefore  we obtain from \eqref{030} that $\mu_{k-1}$ does not vanish either.

Now, \eqref{030} implies that, apart from the nodes $A$ and $C$, the set $\widetilde\ell_1$ can only contain zeros of the polynomial $p^\star_{B,\Bset}\in\Pi_{n-k}$ on the line $\ell_1$ and, therefore, it can contain no more than $n-k$ other nodes.

On the other hand the line $\ell_1$ is a component of the maximal curve $\mu$ of degree $k.$ 
Consequently, according to Proposition \ref{maxcomp}, for $m=1$, in addition to $A$ and $C$, the set $\widetilde\ell_1$ must contain at least $d(n-k+1,1)-2=n-k+2-2=n-k$ nodes of $\Xset\setminus \mu_{k-1}.$

Thus we conclude that the polynomial $p^\star_{B,\Bset}$ has exactly $n-k$ distinct real zeros on the line $\ell,$ which are different from the nodes $A$ and $C,$ and all these $n-k+2$ nodes constitute the set $\widetilde\ell_1.$ Thus the condition (ii) is proved.

Next, let us verify that the conditions (i) and (ii) imply the conditions (iii)- (viii).

From (ii), in view of Proposition \ref{maxcomp}, with $m=1,$ we obtain that the curve $\mu_{k-1}$ is a maximal curve of degree $k-1.$ Then, Corollary \eqref{maxmax} implies that $\Bset_{n-k+1}$ is an  $(n-k+1)$-correct set. Hence $\ell_1$ is a maximal line there. The remaining equalities in (iv) and (v) follow from the definitions of $\mu_{k-1}$ and ${\widetilde\ell}_1.$  Thus, the conditions (i) and (ii) imply (iii), (iv) and (v). 

To verify the condition  (vi) it suffices to prove that $p^\star_{B,\Bset}\in\Pi_{n-k}$ is the 
associated maximal curve of the triplet $\{A,B,C\}$ in  $B_{n-k+1}=\Bset\cup \widetilde\ell_1:$
\begin{equation}\label{1.2.3}B_{n-k+1}\setminus p^\star_{B,\Bset}=\{A,B,C\}.
\end{equation}
Note that this readily follows from the  condition (ii).

Next, let us verify the condition  (vii). Assume that $\{A',B,C'\}\subset\Xset$  is a special triplet, and $A'\in \widetilde\ell_1$ and therefore $A'\in\Aset.$ By using the condition (ii), 2), for this special triplet we obtain that $p^\star_{B,\Bset}(A')\neq 0.$

Note that this last condition, combined with the fact that $A' \in \widetilde{\ell}_1$,
implies $A' = A$ or $A' = C$. Hence, in view of  Proposition \ref{sptr2}, we get $\{A',B,C'\}=\{A,B,C\}.$

Let us mention that the condition (viii)  follows from Corollary \ref{maxmax}.

Finally, let us show that the conditions (i), (ii) imply that the triplet $\{A,B,C\}$ is special in $\Xset.$ Indeed, in view of the first equality in (iv), the condition (ii) and the relation \eqref{1.2.3} imply  that the polynomial 
$$p^\star_{B,\Bset}\ \mu_{k-1}\in \Pi_{n-1}$$
 vanishes at all the nodes of the set $\Xset\setminus\{A,B,C\}.$
On the other hand it is easily seen that this polynomial does not vanish at the nodes $\{A,B,C\}.$
Therefore  $\{A,B,C\}$ is a special triplet in $\Xset.$
\end{proof}

\subsection{The case of several special triplets}

Next, let us consider the case when there are two or more special triplets with a common node. 

\begin{theorem}\label{prp:m2ket}
Let $\Xset,\ \mu:=\mu_k,\ \Bset:=\Bset_{n-k},\ \Aset:=\Aset_{k},$ be as in Theorem \ref{th1sptr}. Suppose that the distinct triplets $\{A_i,B,C_i\},\  1\le i\le m,$ with a common node $B,$ are special, where $m\ge 1, A_i \in \Aset, B\in \Bset,$ and $C_i \in \Xset.$ 

Then, the following conditions hold:
\begin{enumerate}
\setlength{\itemsep}{0mm}
\item
The lines $\ell_i:=\ell_{A_iC_i}$ are distinct components of $\mu=\mu_{k},$ in particular  $B\notin\ell_i,\ C_i\in\Aset,\ i=1,\ldots,m,$  and $m\le k.$ Moreover, we have that
$$A_j,C_j\notin\ell_i \ \forall\ j\neq i,\ 1\le j\le m.$$
\item
The curve $\mu_{k-m}:=\frac{1}{\ell_1\cdots\ell_m}\ \mu$  is a maximal curve of degree $k-m.$  
\item
The set $$\widetilde\ell_i:=\{M\in{\ell_i}\cap  \Xset: \mu_{k-m}(M)\neq 0\},\ i=1,\ldots,m,$$ contains exactly $(n-k)+(m-1)+2$ nodes of $\Bset_{n-k+m}:=\mathcal X\setminus \mu_{k-m}.$ Namely, it contains the following distinct nodes:

1) the nodes $A_i$ and $C_i,$ 

2) $n-k$ nodes of $\Xset$ coinciding with the zeros of $p^\star_{B,\Bset}\in\Pi_{n-k}$ on the line $\ell_i,$ which all are therefore real and distinct,

3) $m-1$ nodes of intersection: 
$$D_{ij}:=\ell_i\cap\ell_j\ \ \forall\ j\neq i,\ 1\le j \le m,$$
which all are therefore different.  

\item
The set  
$\Bset_{n-k+m}:=\Xset\setminus\mu_{k-m}=\Bset\cup\widetilde\ell_1\cup\cdots\cup\widetilde\ell_m
$ 

is an  $(n-k+m)$-correct set, for which $\ell_1,\ldots,\ell_m$ are maximal lines.
\item
$\Aset_{k-m}:=\Xset\cap\mu_{k-m}=\Aset\setminus (\widetilde\ell_1\cup\cdots\cup\widetilde\ell_m).$ 
\item 
The triplet $\{A_i, B, C_i\}$, for each $1 \le i \le m$, is also a special triplet in the $(n-k+m)$-correct set $\Bset_{n-k+m}$, with the associated maximal curve $$p^\star_{B,\Bset}\prod_{j=1, j\neq i}^m\ell_j.$$
\item
If $\Xset$ is a $GC_n$-set then $\Bset_{n-k+i}$ is a $GC_{(n-k+i)}$-set, where $1\le i\le m.$
 \end{enumerate}
\end{theorem}
\begin{proof}
Let us apply Theorem \ref{th1sptr} to an arbitrary fixed triplet $\{A_{i_1},B,C_{i_1}\}$, where $1\le i_1\le m.$
Then we obtain that 
the line $\ell_1:=\ell_{A_{i_1}C_{i_1}}$ is a component of $\mu,$ in particular $B\notin \ell_1$ and $C_{i_1}\in \Aset.$  

Now, let us verify that for any $j\neq i_1$, with $1\le j\le m$, we have that
\begin{equation}\label{setk}A_j \in \Aset_{k-1}=\Aset_k\setminus \widetilde\ell_1.
\end{equation} 
Indeed, assume conversely that $A_j\in\widetilde\ell_1.$ Then, according 
Theorem \ref{th1sptr}, (vii), we get that $\{A_j,B,C_j\}=\{A_{i_1},B,C_{i_1}\},$
which is a contradiction. Thus the statement (i) together with the part ``Moreover" is established.

Let us now proceed by induction on $m.$

The case $m=1$ reduces to Theorem \ref{th1sptr}. Assume that Theorem \ref{prp:m2ket} holds for $m,$ and prove it for $m+1.$

Now, in view of \eqref{setk}, we have that
\begin{equation*}A_{m+1} \in \Aset_k\setminus \widetilde\ell_i, \ i=1,\ldots,m.
\end{equation*} 
This implies evidently $A_{m+1} \in \Aset_{k-m}=\Aset\setminus (\widetilde\ell_1\cup\cdots\cup\widetilde\ell_m).$

Consequently, given the inductive hypothesis, we can apply Theorem \ref{th1sptr} to the maximal curve  $\mu_{k-m}:=\frac{1}{\ell_1\cdots\ell_m}\ \mu_k$ as well as to the corresponding sets $\Bset_{n-k+m}$ and $\Aset_{k-m}$, and to the triplet $\{ A_{m+1},B,C_{m+1}\}$.

 Then we obtain that 
the line $\ell_{m+1}:=\ell_{A_{m+1}C_{m+1}}$ is a component of $\mu_{k-m}$ and therefore a component of $\mu_k.$ In particular $B\notin \ell_{m+1}$ and $A_{m+1},C_{m+1}\in \Aset_{k-m}.$   Next, the curve $\mu_{k-m-1}:=\frac{1}{\ell_{m+1}}\ \mu_{k-m}=\frac{1}{\ell_1\cdots\ell_{m+1}}\ \mu$  is a maximal curve of degree $k-m-1.$  In view of the induction hypothesis, this proves the condition  (ii).

Therefore, according to Theorem \ref{th1sptr}(ii), the set $$\widetilde\ell_{m+1}:=\{M\in{\ell_{m+1}}\cap  \Xset: \mu_{k-m-1}(M)\neq 0\}$$ contains the nodes $A_{m+1}$ and $C_{m+1},$ 
and exactly $(n-k+m)$ other nodes
 of $\mathcal X\setminus \mu_{k-m},$ which are the zeros of $p^\star_{B,\Bset_{n-k+m}}$ on the line $\ell_{m+1}.$ Now, note that
\begin{equation}\label{pstar}p^\star_{B,\Bset_{n-k+m}}=p^\star_{B,\Bset}\ell_1\cdots\ell_m.\end{equation}
Thus, more precisely, in addition to the nodes $A_{m+1}$ and $C_{m+1}$, the set $\widetilde\ell_{m+1}$ contains the following $(n-k+m)$ nodes:

1) $n-k$ nodes that are the zeros of $p^\star_{B,\Bset}\in\Pi_{n-k}$ on the line $\ell_{m+1}$, and

2) $m$ intersection nodes:
$$D_{i,m+1}:=\ell_i\cap\ell_{m+1},\ i=1,\ldots,m.$$
This proves  the statement (iii).

Next, by  Theorem \ref{th1sptr} (iv) and induction hypothesis we get 
$$\Bset_{n-k+m+1}:=\Xset\setminus\mu_{k-m-1}=\Bset_{n-k+m}\cup\widetilde\ell_{m+1}=\Bset\cup\widetilde\ell_1\cup\cdots\cup
\widetilde\ell_{m+1}.
$$
In view of Corollary \eqref{maxmax}, the set  $\Bset_{n-k+m+1}$
is an $(n-k+m+1)$-correct set for which $\ell_{m+1}$, just like each of the lines $\ell_i$ (for $1 \le i \le m$), is a maximal line, given that the numbering of the lines does not matter. Thus the statement (iv) is verified.
Similarly the statement (v) is verified. 

The statement (vi) follows from   Theorem \ref{th1sptr} (vi) and the relation 
\eqref{pstar}.
Finally, the condition (viii) follows directly from Corollary \ref{maxmax}.
\end{proof}

\section{Some results for $2$-node lines}
Now, let us consider the case when the set $\Xset$ is a $GC_n$-set and the special triplets are distinct and originate from $2$-node lines.
\begin{theorem}\label{avel+}
Let $\Xset$ be a $GC_n$-set. Let also $\mu:=\mu_k,\ \Bset:=\Bset_{n-k},\ \Aset:=\Aset_{k},\  \Aset_i,\  \Bset_i$ be as in Theorem \ref{prp:m2ket}. 
Assume that the distinct special triplets $\{A_i,B,C_i\},$ $1\le i\le m,$ with a common node $B,$ originate from $2$-node lines $\ell_{A_i,B},$ used by $C_i,$ respectively, where $2\le m\le k,\ A_i \in \Aset,\ B\in \Bset,$ and $C_i \in \Xset.$ 

Then the set $\Bset_{n-k+m}$ is a $GC_{n-k+m}$ set, for which each line $\ell_i:=\ell_{A_iC_i},\ 1\le i\le m$ is a maximal line.
Moreover, the nodes $B,C_1,\ldots, C_m$ are collinear and give one more maximal line of $\Bset_{n-k+m}: \ \ell_{BC_1}.$  Thus the set $\Bset_{n-k+m}$ has $m+1$ maximal lines.
\end{theorem}
\begin{proof}
We need to verify only the statement ``Moreover" since the others follow from Lemma \ref{cor:2ket} and Theorem \ref{prp:m2ket}. To do this, it suffices to consider the case $m=2.$

Suppose that $\Xset$ is a $GC_n$-set and the special triplets $\{A_i, B, C_i\}$ originate from the 2-node lines $\ell_{A_iB},$ used by $C_i,$ respectively, where $i=1,2.$ Let us prove that the three nodes $B,\ C_1$ and $C_2$ are collinear. 

Denote $\ell_i=\ell_{A_iC_i},\ i=1,2.$

Consider the special case of Theorem \ref{prp:m2ket} when $k=n$ and hence the maximal curve of degree $n$ is a fundamental polynomial of $B$:   
$\mu=\mu_n=p^\star_{B,\Xset}.$ In this case we have also $\Bset=\Bset_0=\{B\}.$

From the statements (iii) and (iv) in this case we obtain that  
$$\Bset_2=\Bset_0\cup\widetilde\ell_1\cup\widetilde\ell_2=\{B,A_1,C_1,A_2,C_2,D_{12}\}$$ 
is a 
$GC_2$-set, where $D_{12}=\ell_1\cap\ell_2.$ 
Now, let $\ell$ be a $3$-node (maximal) line the node $D_{12}$ uses in $\Bset_2.$ First, let us prove that $B\in \ell.$ Assume, by way of contradiction, that $B\notin \ell.$ Then, we get that three nodes from the set  $\{A_1,C_1,A_2,C_2\}$ are collinear. Therefore the three nodes contain either the couple $A_1,C_1$  or $A_2,C_2.$ In other words $\ell$ coincides either with the line $\ell_1$ or $\ell_2.$ This is a contradiction since, the intersection node $D_{12}$ belongs to both mentioned lines.

Hence we have that $B\in \ell.$ Then, since $\ell_{A_1B}$ and   $\ell_{A_2B}$ are $2$-node lines we conclude that $A_1,A_2\notin \ell.$ Therefore $C_1,C_2\in \ell.$
\end{proof}

\begin{corollary} \label{cor2}
Let $\Xset$ be an $n$-correct set. Assume that $2$-node lines: $\ell_{A_1B}$ and $\ell_{A_2B},$ with a common node $B,$
are used by nodes $C_1$ and  $C_2$, respectively, where $A_i,B,C_i\in \Xset.$  
Then we have that  $C_1\neq C_2.$ 

Moreover,  the following holds: 
\begin{equation}\label{2=}C_1=A_2\ \iff \ \{A_1,B,C_1\} = 
\{A_2,B,C_2\}\ \iff \ C_2= A_1.\end{equation}

Furthermore, in  the case $C_1\neq A_2,$ and therefore $\{A_1,B,C_1\} \neq \{A_2,B,C_2\}$ and $C_2\neq A_1,$ the following conditions take place: 
\begin{enumerate}
\setlength{\itemsep}{0mm}
\item
The lines $\ell_1:=\ell_{A_1C_1}$ and $\ell_2:=\ell_{A_2C_2}$ are different components of $p^\star_{B,\Xset}.$
\item
The curve
$\mu_{n-2}:=\frac{1}{\ell_1 \ell_2}p^\star_{B,\Xset}$
is a maximal curve of degree $n-2.$
\item
The set $\Bset_2:=\Xset\setminus \mu_{n-2}=\{A_1,A_2,B,C_1,C_2,D_{12}\}$ is a $2$-correct set, where $\ell_1, \ell_2$ are maximal lines and 
\begin{equation}\label{C=cap22}D_{12}=\ell_1\cap \ell_2.\end{equation}
\item
If $\Xset$ is a $GC_n$-set then $\Bset_2$ is a $GC_2$-set and the nodes $B,C_1,C_2$ are collinear. 
\end{enumerate}
\end{corollary}

\begin{proof}
Note that the statement $C_1\neq C_2$ follows from Corollary \ref{2nodeuse} and the statement ``Moreover" follows from 

Then let us verify the first equivalence in \eqref{2=}. The second one can be done similarly. The side $\Rightarrow$  follows from  Lemma \ref{cor:2ket} and Proposition \ref{sptr2}. For the side $\Leftarrow$ note that we have that $C_1 \neq C_2$ and $C_1\neq B.$ Consequently the equality of the triplets implies that $C_1=A_2.$

In other words, here we have description of the case of two distinct $2$-node lines $\ell_{A_1B}$ and $\ell_{A_2B},$ (used by $A_2$ and $A_1,$ respectively) which originate a single triplet, namely  $\{A_1,B,A_2\}.$

Next, the relation \eqref{2=} implies that the two triplets mentioned in the statement following the phrase "Furthermore" are, in fact, distinct. Consequently, in this case, all that remains is to apply Theorem \ref{avel+} in the special case $m=2$ and $k=n.$
\end{proof}

In Fig. \ref{Fig4}, we present a subset $\Bset_2\in\Xset$ of Corollary \eqref{cor2} in two cases. Case (a): the $\Xset$ is an $n$-correct set, and case (b): the $\Xset$ is a $GC_n$-set. Note that if $n=2$ then $\Xset=\Bset_2.$ 

Note also that, in each set  in Fig. \ref{Fig4} we have two used $2$-node lines $\ell_{A_iB},\ i=1,2,$ with a common node $B,$ which correspond to two distinct special triplets: $\{A_i,B,C_i\},\ i=1,2,$ respectively.

On the other hand, in Fig. \ref{Fig6} we have two used $2$-node lines $\ell_{A_iB},\ i=1,2,$ with a common node $B,$ in a $3$-correct set, which correspond to a single special triplet: $\{A_1,B,A_2\}.$ Indeed, $A_2$ uses $\ell_{A_1B}$ and $A_1$ uses $\ell_{A_2B},$ since $p^\star_{A_1}=\ell_{A_2B}\ell_1\ell_2$ and $p^\star_{A_2}=\ell_{A_1B}\ell_1\ell_2.$   

Next, we show that to any set of $m$ distinct used $2$-node lines there corresponds a set of $m$ distinct special triplets, provided that $m \ge 3$.
\begin{proposition} \label{n>2}
Let $\Xset$ be a $GC_n$-set and $m\ge 3.$
Let also distinct $2$-node lines $\ell_{A_i,B},\  1\le i\le m,$  be used by the nodes $C_i,$ respectively, where $A_i,B,C_i\in\Xset.$ Then the corresponding special triplets $\{A_i,B,C_i\},\  1\le i\le m,$ are always distinct.
\end{proposition}
\begin{figure}
     \centering
     \begin{subfigure}[b]{0.4\textwidth}
         \centering
         \includegraphics[width=\textwidth]{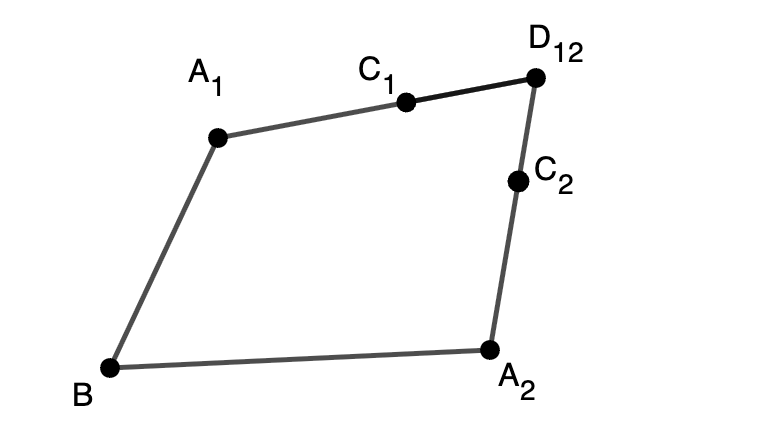}
         \caption{$\Bset_2$ is a $2$-correct set}
         \label{fig.3a}
     \end{subfigure}
     \hfill
     \begin{subfigure}[b]{0.4\textwidth}
         \centering
         \includegraphics[width=\textwidth]{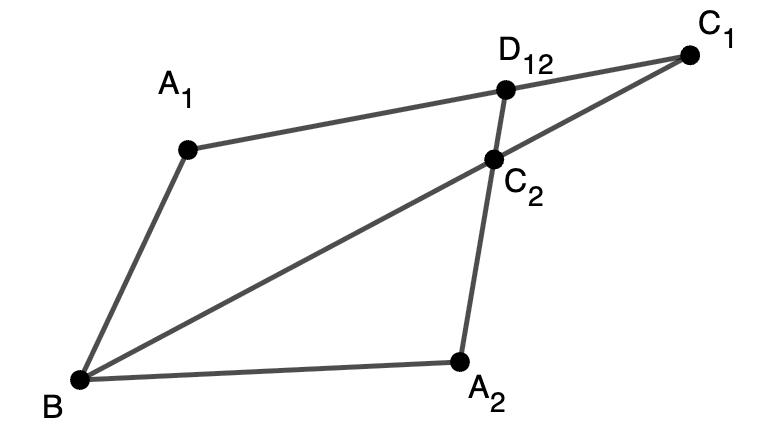}
         \caption{$\Bset_2$ is a $GC_2$-set}
         \label{fig.3b}
     \end{subfigure}
            \caption{The set $\Bset_2$}
        \label{Fig4}
\end{figure}

\begin{proof}
Without loss of generality we may assume that $m=3.$ Consider three distinct used $2$-node lines: $\ell_{A_1B}, \ell_{A_2B}, \ell_{A_3B},$ with a common node $B,$ where $A_1,A_2,A_3,B\in \Xset.$
Assume, that these lines are
used by the nodes $C_1, C_2$ and  $C_3$, respectively, where $C_i\in \Xset.$

Now, assume by way of contradiction that two triplets, say $\{A_1,B,C_1\}$ and $\{A_2,B,C_2\}$ coincide. As was established in Corollary \ref{cor2},  this means that 
\begin{equation}\label{A1C2} A_1=C_2\ \hbox{and}\ A_2=C_1.
\end{equation}
Note that the nodes  $C_1,C_2$ and $C_3$ are distinct  since, in view of Corollary \ref{2nodeuse}, a node cannot use two different $2$-node lines. 

Then, in view of \eqref{A1C2}, we have that $C_3$ is different from $A_1.$ 

Since $C_3\neq A_1$ we get from Corollary \ref{cor2}, that the triplets $\{A_1,B,C_1\}$ and $\{A_3,B,C_3\}$ are distinct. Hence, Theorem \ref{avel+} yields that the nodes $B, C_1, C_3$ are collinear.

Therefore we have that $C_3\in\ell_{BC_1}\equiv\ell_{BA_2}.$ Hence, the line $\ell_{BA_2}$ is not a $2$-node line, which is a contradiction.
\end{proof}
\begin{corollary}
\emph{(i)} Let $\Xset$ be an $n$-correct set, $n\ge 2.$ Then there are at most $n$ special triplets with a fixed common node $B\in \Xset.$ In the case there are $n$ such triplets we have exactly $n$ maximal lines in $\Xset$ not passing through $B.$

\emph{(ii)} Let $\Xset$ be a $GC_n$-set, $n\ge 3.$ Then there are at most $n$ used $2$-node lines with a fixed common node $B\in\Xset.$ In the case there are $n$ such lines we have exactly $n+1$ maximal lines in $\Xset,$ from which $n$ not passing through $B$ and a maximal line passing through $B.$ Thus, the set $\Xset$ is a Carnicer-Gasca set.
\end{corollary}
\begin{figure}
\begin{center}
\includegraphics[width=6.0cm,height=3.0cm]{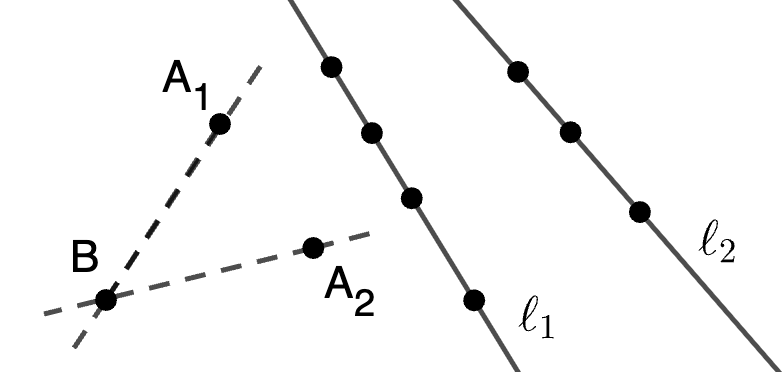}
\end{center}
\caption{Two used $2$-node lines correspond to a single special triplet} \label{Fig6}
\end{figure}

\begin{proof}
(i) Let us apply Theorem \ref{prp:m2ket}, (i), with the maximal curve $\mu=p^\star_{B,\Xset}$ of degree $n.$ Then for the number $m$ of distinct special triplets with a fixed common node $B$ we have the inequality $m\le n.$

In the case there are 
$n$ such triplets we obtain from Theorem \ref{prp:m2ket}, (vi), that there are $n$ maximal lines: $\ell_1,\ldots,\ell_n$ in the $n$-correct set $\Bset_n\equiv \Xset$ not passing through $B.$ 

Note that there is no other maximal line that does not pass through the node $B.$ Indeed, in view of Proposition \ref{prp:n+1ell}, this would contradict the existence of the fundamental polynomial $p^\star_B.$

(ii) In this case, it follows from Proposition \ref{n>2} that the triplets corresponding to the used $2$-node liness are distinct.Then Theorem \ref{avel+} implies that there are $n+1$ maximal lines in $\Xset,$ described in (ii). 

Finally, we readily obtain that there are exactly $n+1$ maximal lines in $\Xset.$ Indeed, otherwise if there are $n+2$ such lines then $\Xset$ becomes a Chung-Yao set, where the only used lines are the maximal lines and therefore no $2$-node line is used. 
\end{proof} 
In Fig. \ref{Fig1} we have a $GC_3$ set, which is a Carnicer-Gasca set and has three used $2$-node lines with a common node $B$. We have also exactly $3$ maximal lines: $\ell_i,\ i=1,2,3,$ not passing through $B$ and a maximal line: $\ell_0,$ that passes through $B.$

At the end let us formulate two results concerning the Gasca-Maeztu conjecture for $GC_6$-sets.
\begin{corollary}\label{6655} Let $\Xset$ be a $GC_6$-set without a maximal line. 
Suppose that there are two  $2$-node lines: $\ell_{A_iB},\ i=1,2,$  with a common node $B,$ used by $C_i,$ respectively, where $A_i,B,C_i\in \Xset.$ Then one of the following conditions is satisfied:
\begin{enumerate}
\item
(i) the two $2$-node lines correspond to a single 
special triplet, i.e., $C_1=A_2$ and therefore $C_2=A_1,$
\item
(ii) There are two $6$-node lines that do not intersect at the nodes of $\Xset$, and the remaining $10$ nodes of $\Xset$ are distributed into two lines of five.
\item
(iii) there are three $6$-node lines that do not intersect at the nodes of $\Xset$,
\end{enumerate}

\end{corollary}
\begin{proof} Suppose that the two $2$-node lines  correspond to two distinct special triplets, i.e., $C_1\neq A_2$ and therefore $C_2\neq A_1.$ Then, according to  Theorem \ref{prp:m2ket}, we obtain that the curve 
$$\mu_{4}:=\frac{1}{\ell_1 \ell_2} p^\star_{B,\Xset}$$
is a maximal curve of degree $4,$ where $\ell_i:=\ell_{AiC_i}.$ Hence $\mu_4$ passes through $d(6,4)=22$ nodes and is a product of four distinct lines. Since these lines are not maximal, we readily get that either three or two off them pass through $6$ nodes and non of these nodes can be common for two or three of them. It remains to note that if there are only two $6$-node lines then the remaining $10$ nodes of $\Xset$ are distributed into two lines of five.
\end{proof}

\begin{corollary}\label{666} Let $\Xset$ be a $GC_6$-set without a maximal line. 
Suppose that there are three  $2$-node lines: $\ell_{A_iB},\ i=1,2,3,$  with a common node $B,$ used by $C_i,$ respectively, where $A_i,B,C_i\in \Xset.$ Then there are three $6$-node lines that do not intersect at the nodes of $\Xset$.
\end{corollary}
\begin{proof} According to Proposition \ref{n>2}, (iii),  the three special triplets corresponding to the three used $2$-node lines are distinct. Thus, by using Theorem \ref{prp:m2ket}, we obtain that the curve 
$$\mu_{3}:=\frac{1}{\ell_1 \ell_2 \ell _3}p^\star_{B,\Xset}$$
is a maximal curve of degree $3,$ where $\ell_i:=\ell_{AiC_i}.$ Hence $\mu_3$ passes through $d(6,3)=18$ nodes and is a product of three distinct lines. Since these lines are not maximal, we readily get that each of them passes through $6$ nodes and non of these nodes can be common for two or three of them.
\end{proof}


\end{document}